\documentclass[11pt,reqno]{article}

\usepackage[utf8]{inputenc}
\usepackage[T1]{fontenc}
\usepackage{lmodern}
\usepackage{amsmath,amssymb,amsfonts,amsthm,mathrsfs,mathtools}
\usepackage{cite}
\usepackage{graphicx}
\usepackage{float}
\usepackage{hyperref}
\hypersetup{pdftitle={Spectral Limits of the Neumann-Poincare Operator on Degenerating Tori},pdfauthor={Doosung Choi},hidelinks}

\newtheorem{theorem}{Theorem}[section]
\newtheorem{lemma}[theorem]{Lemma}
\newtheorem{prop}[theorem]{Proposition}
\theoremstyle{remark}

\newcommand{\RR}{\mathbb R}
\newcommand{\ZZ}{\mathbb Z}
\newcommand{\Kcal}{\mathcal K}
\newcommand{\Lcal}{\mathcal L}
\newcommand{\Scal}{\mathcal S}
\newcommand{\Acal}{\mathcal A}
\newcommand{\Id}{\operatorname{Id}}

\numberwithin{equation}{section}
\allowdisplaybreaks

\title{Spectral Limits of the Neumann--Poincar\'e Operator on Degenerating Tori}

\author{Doosung Choi\thanks{Department of Mathematics, Applied Mathematics,
and Statistics, Case Western Reserve University, Cleveland, OH 44106,
USA (\href{mailto:doosung.choi@case.edu}{doosung.choi@case.edu})}}

\date{}

\begin{document}

\maketitle

\begin{abstract}
We determine the limit of the spectrum of the Neumann--Poincar\'e operator on circular tori as the torus degenerates. As the central hole closes, the spectrum converges to $[-1/2,0]$ together with a set of positive points, which we characterize as eigenvalues of explicit limiting operators. As the cross-section shrinks, the spectrum converges to $[0,1/2]$.
\end{abstract}



\section{Introduction}

The Neumann--Poincar\'e (NP) operator is a boundary integral operator that arises in the layer potential approach to boundary value problems for the Laplace equation \cite{Kellogg,Verchota}. For a domain with a $C^{1,\alpha}$ boundary, it is a compact operator that becomes self-adjoint with respect to the inner product induced by the single-layer potential, and its spectrum is contained in $[-1/2,1/2]$ (see \cite{KhavinsonPutinarShapiro,AndoKangMiyanishiPutinar}). Its eigenvalues are the critical values of Poincar\'e's variational problem \cite{KhavinsonPutinarShapiro}, and they govern the solvability of conductivity and dielectric transmission problems \cite{KangKimLeeShinYu2016}, the quasi-static plasmon resonances of small particles \cite{AndoKang2016,AmmariMillienRuizZhang2017}, the anomalous localized resonance and cloaking \cite{AmmariCiraoloKangLeeMilton2013,ChungKangKimLee2014}, and the shape reconstruction from plasmonic resonances \cite{AmmariPutinarRuizYuZhang2019}.

The spectrum of the NP operator is explicitly known only for a few shapes. In two dimensions, it is known for disks and ellipses \cite{KhavinsonPutinarShapiro}, confocal ellipses \cite{ChungKangKimLee2014}, concentric annuli \cite{ChoiLimShipman}, and two disks of equal radius \cite{KangKimLeeShinYu2016,AmmariPutinarRuizYuZhang2019}, and a matrix representation in terms of the Grunsky coefficients is available for smooth simply connected domains \cite{JungLimFaber2021}. In three dimensions, it is known for spheres, spheroids, and ellipsoids \cite{KhavinsonPutinarShapiro,FengKang2016,AndoKangMiyanishiThinEllipsoids}, Weyl-type asymptotics of the eigenvalues were obtained in \cite{MiyanishiRozenblum,Miyanishi}, and axially symmetric surfaces were studied in \cite{FukushimaKang}. Boundaries with corners or cusps produce essential spectrum, as for intersecting disks \cite{KangLimYu}, touching disks and crescent domains \cite{JungLim2023}, curvilinear polygons \cite{PerfektPutinarCorners2017}, rotationally symmetric conical points \cite{HelsingPerfekt2018}, and polyhedra \cite{DeLeonPerfekt2023}. We refer to \cite{AndoKangMiyanishiPutinar} for more results.

For degenerating domains, the eigenvalues on thin planar rectangles become dense in $[-1/2,1/2]$ (see \cite{AndoKangMiyanishiThin2D}), those on elongated prolate spheroids fill $[0,1/2]$, and those on flat oblate ellipsoids fill $[-1/2,1/2]$ (see \cite{AndoKangMiyanishiThinEllipsoids}). In \cite{ChoiLimShipman}, together with Lim and Shipman, we proved that the spectrum for thin doubly connected planar domains converges to $[-1/2,1/2]$ in the Hausdorff distance. Near-contact asymptotics of two planar inclusions were studied in \cite{BonnetierTriki2013}.

In this paper, we consider circular tori. Ando et al.\ \cite{AndoJiKangKawagoeMiyanishi} showed that the NP operator on every torus has infinitely many positive and negative eigenvalues. For thin rings, Majic \cite{Majic2019} derived the electrostatic $T$-matrices, and Ruiz and Schnitzer \cite{RuizSchnitzer2022} obtained asymptotic formulas for the plasmonic eigenvalues at fixed Fourier modes (see also \cite{Love1975,Garapati2017}). For the elastic NP operator on tori, Tang \cite{Tang2026} determined the essential spectrum of each Fourier block. A torus degenerates either by closing its central hole, which produces a horn torus, or by shrinking its cross-section. We determine the limit of the whole NP spectrum in both cases. A fixed-mode analysis is not sufficient for this, since Fourier modes whose indices grow with the degeneration contribute to the spectrum.

In Theorem~\ref{thm:horn_full_limit}, we show that, as the hole closes, the spectrum converges in the Hausdorff distance to the union of $[-1/2,0]$ and an infinite set of positive points accumulating only at $0$, and that the positive eigenvalues converge with multiplicity. The positive points are the positive eigenvalues of explicit limiting operators on the real line, one for each Fourier mode. Each limiting operator is the sum of a nonpositive multiplication operator, whose range produces the interval $[-1/2,0]$, and a positive integral operator, which produces the positive points. Theorem~\ref{thm:slender_ring_limit} shows that, as the cross-section shrinks, the spectrum converges to $[0,1/2]$.

Our main tool is a matrix representation of the NP operator in terms of toroidal harmonics, based on the Fourier decomposition in \cite{AndoJiKangKawagoeMiyanishi}. Each Fourier block is the sum of a diagonal matrix and a positive weighted Toeplitz matrix, whose entries are products of associated Legendre functions. In the horn-torus limit, these functions converge to modified Bessel functions, and uniform bounds in both Fourier indices allow us to control all blocks at once. The thin-ring limit is reduced to the horn-torus case by a duality between the matrix entries.

The paper is organized as follows. Section~\ref{sec:prelim} reviews the layer potentials and the toroidal harmonics, and section~\ref{sec:representation} gives the matrix representation. Section~\ref{sec:asymptotics} derives the Bessel asymptotics of the matrix entries. The horn-torus and thin-ring limits are proved in sections~\ref{sec:spectral_consequences} and \ref{sec:thin_ring}, and section~\ref{sec:numerical_validation} presents numerical examples. Appendix~\ref{app:weight_bounds} contains the proof of Lemma~\ref{lem:green_weight_bounds}.

\section{Preliminary}
\label{sec:prelim}

\subsection{Layer potentials}
\label{subsec:layer_toroidal}

Let $\Omega$ be a bounded domain in $\RR^3$ with $C^{1,\alpha}$ boundary for some $0<\alpha<1$. The NP operator associated with $\Omega$ is defined by
\[
\Kcal_{\partial\Omega}^*[\varphi](x)
=\frac1{4\pi}\,\mathrm{p.v.}\!\int_{\partial\Omega}
\frac{\langle x-y,\nu_x\rangle}{|x-y|^3}\,\varphi(y)\,dS(y),
\qquad x\in\partial\Omega,
\]
where $\nu_x$ is the outward unit normal vector at $x$ and p.v.\ stands for the Cauchy principal value. The single-layer potential is
\[
\Scal_{\partial\Omega}[\varphi](x)
=-\frac1{4\pi}\int_{\partial\Omega}\frac{\varphi(y)}{|x-y|}\,dS(y),
\qquad x\in\RR^3,
\]
and the two operators are related by the jump relation \cite{Verchota,KhavinsonPutinarShapiro}
\begin{equation}\label{jump_relation}
\partial_\nu\Scal_{\partial\Omega}[\varphi]\big|_\pm
=\Bigl(\pm\frac12\Id+\Kcal_{\partial\Omega}^*\Bigr)[\varphi]
\quad\text{on }\partial\Omega,
\end{equation}
where $\partial_\nu=\nu\cdot\nabla$ and the symbols $+$ and $-$ indicate the limits from the exterior and the interior of $\Omega$, respectively. We denote by $\mathscr H^*(\partial\Omega)$ the space $H^{-1/2}(\partial\Omega)$ equipped with the inner product
\[
\langle\varphi,\psi\rangle_{\mathscr H^*(\partial\Omega)}
:=-\bigl\langle\varphi,\Scal_{\partial\Omega}[\psi]\bigr\rangle_{-1/2,1/2},
\]
where $\langle\cdot,\cdot\rangle_{-1/2,1/2}$ is the duality pairing between $H^{-1/2}(\partial\Omega)$ and $H^{1/2}(\partial\Omega)$, linear in the first and conjugate linear in the second argument. The norm induced by this inner product is equivalent to the usual $H^{-1/2}$ norm, and $\Kcal^*_{\partial\Omega}$ is a compact self-adjoint operator on $\mathscr H^*(\partial\Omega)$ whose spectrum $\sigma(\Kcal^*_{\partial\Omega})$ is contained in $[-1/2,1/2]$ (see \cite{KhavinsonPutinarShapiro,AndoKangMiyanishiPutinar,KangKimLeeShinYu2016}).

\subsection{Toroidal coordinates and harmonics}
\label{subsec:toroidal_coordinates}

For $a>0$, the toroidal coordinates $(\tau,\phi,\theta)\in(0,\infty)\times(\RR/2\pi\ZZ)^2$ are defined by \cite{Basset1893}
\[
x=\frac{a\sinh\tau\cos\phi}{\cosh\tau-\cos\theta},\quad
y=\frac{a\sinh\tau\sin\phi}{\cosh\tau-\cos\theta},\quad
z=\frac{a\sin\theta}{\cosh\tau-\cos\theta}.
\]
The angle $\phi$ describes the rotation about the $z$-axis, and $\theta$ parametrizes the circular cross-section. The surface $\{\tau=\tau_0\}$ is the boundary of the solid torus $\Omega_0=\{\tau>\tau_0\}$, whose major and minor radii are $R_0=a\coth\tau_0$ and $r_0=a\operatorname{csch}\tau_0$. We set $a=R_0\tanh\tau_0$ so that $R_0$ is fixed and $r_0=R_0\operatorname{sech}\tau_0$. Since the NP spectrum is invariant under dilations, it depends only on $\tau_0$. The central hole closes as $\tau_0\downarrow0$, and the cross-section shrinks as $\tau_0\to\infty$. On $\partial\Omega_0$, we have
\begin{equation}\label{normal_derivative}
dS=\frac{a^2\sinh\tau_0}{(\cosh\tau_0-\cos\theta)^2}\,d\phi\,d\theta,
\qquad
\partial_\nu=-\frac{\cosh\tau_0-\cos\theta}{a}\,\partial_\tau\Big|_{\tau=\tau_0}.
\end{equation}
Note that the minus sign in $\partial_\nu$ appears because $\tau$ increases into $\Omega_0$.

Let $P^\mu_\xi$ and $Q^\mu_\xi$ be the associated Legendre functions of the first and second kinds on $(1,\infty)$, with the standard branches as in \cite[\S14.3(ii)]{DLMF}. For $m,n\in\ZZ$, we define the toroidal harmonics
\[
G_{mn}=\sqrt{\cosh\tau-\cos\theta}\;Q^{|m|}_{|n|-\frac12}(\cosh\tau)\,e^{im\phi+in\theta},
\]
\[
H_{mn}=\sqrt{\cosh\tau-\cos\theta}\;P^{|m|}_{|n|-\frac12}(\cosh\tau)\,e^{im\phi+in\theta}.
\]
They are harmonic functions \cite{Basset1893}. Here $G_{mn}$ is regular in the interior of the torus, including the focal circle, and $H_{mn}$ is regular in the exterior and is $O(|x|^{-1})$ as $|x|\to\infty$.

We will use the order conversion and the Wronskian formulas \cite[Eqs.~14.9.13, 14.2.10]{DLMF}: for $\eta>1$,
\begin{align}
&P^{|m|}_{|n|-\frac12}(\eta)
=\frac{\Gamma(|n|+|m|+\frac12)}{\Gamma(|n|-|m|+\frac12)}
P^{-|m|}_{|n|-\frac12}(\eta),\label{order_conversion}\\
&Q^{|m|}_{|n|-\frac12}(\eta)\bigl(P^{-|m|}_{|n|-\frac12}\bigr)'(\eta)
-P^{-|m|}_{|n|-\frac12}(\eta)\bigl(Q^{|m|}_{|n|-\frac12}\bigr)'(\eta)
=\frac{(-1)^{|m|}}{\eta^2-1}.\label{normalized_wronskian}
\end{align}
From the integral representations \cite[Eqs.~14.12.7, 14.12.5]{DLMF}, one can easily see that $P^{-|m|}_{|n|-1/2}>0$ on $(1,\infty)$, and hence $P^{|m|}_{|n|-1/2}$ does not vanish there by \eqref{order_conversion}. For $|n|\ge|m|$, we also use the integral representations
\begin{align}
&P^{-|m|}_{|n|-\frac12}(\cosh\tau)
=\frac{(-1)^{|m|}}{2\pi}
\frac{\Gamma(|n|+\frac12)}{\Gamma(|n|+|m|+\frac12)}
\int_0^{2\pi}
\frac{\cos(|m|\vartheta)\,d\vartheta}
{(\cosh\tau+\sinh\tau\cos\vartheta)^{|n|+\frac12}},
\label{P_integral}\\
&Q^{|m|}_{|n|-\frac12}(\cosh\tau)
=(-1)^{|m|}\frac{\Gamma(|n|+\frac12)}{\Gamma(|n|-|m|+\frac12)}
\int_0^\infty
\frac{\cosh(|m|t)\,dt}
{(\cosh\tau+\sinh\tau\cosh t)^{|n|+\frac12}},
\label{Q_integral}
\end{align}
which follow from \cite[Eqs.~14.12.7, 14.9.11, 14.9.13]{DLMF} and from Whipple's formula together with \cite[Eqs.~14.9.16, 14.12.4, 14.3.10]{DLMF}, respectively. The condition $|n|\ge|m|$ ensures the convergence of \eqref{Q_integral}.

\section{Matrix representation of the NP operator on a torus}
\label{sec:representation}

\subsection{Density basis functions}
\label{subsec:density_basis}

For $m,n\in\ZZ$, we define the density functions on $\partial\Omega_0$ by
\begin{equation}\label{constantC}
\varphi_n^m=\kappa_{mn}(\cosh\tau_0-\cos\theta)^{3/2}e^{im\phi+in\theta},
\qquad
\kappa_{mn}=-\frac{(-1)^{|m|}}
{a\sinh\tau_0\,P^{-|m|}_{|n|-\frac12}(\cosh\tau_0)},
\end{equation}
and set
\begin{equation}\label{alpha_def}
b_n^{(m)}(\tau_0)
=(-1)^{|m|}P^{-|m|}_{|n|-\frac12}(\cosh\tau_0)\,Q^{|m|}_{|n|-\frac12}(\cosh\tau_0).
\end{equation}

\begin{prop}\label{prop:density}
For every $m,n\in\ZZ$, we have
\begin{equation}\label{umn_def}
\Scal_{\partial\Omega_0}[\varphi_n^m]
=\begin{cases}
G_{mn},&\tau>\tau_0\ \ (\text{interior}),\\[1mm]
\dfrac{Q^{|m|}_{|n|-\frac12}(\cosh\tau_0)}{P^{|m|}_{|n|-\frac12}(\cosh\tau_0)}\,
H_{mn},&\tau<\tau_0\ \ (\text{exterior}).
\end{cases}
\end{equation}
In particular,
\begin{equation}\label{single_layer_trace}
\Scal_{\partial\Omega_0}[\varphi_n^m]\big|_{\partial\Omega_0}
=-a\sinh\tau_0\,\kappa_{mn}\,b_n^{(m)}(\tau_0)\,
(\cosh\tau_0-\cos\theta)^{1/2}e^{im\phi+in\theta}.
\end{equation}
\end{prop}

\begin{proof}
Let $u$ denote the right-hand side of \eqref{umn_def}. Then $u$ is continuous across $\partial\Omega_0$, harmonic on both sides, and decays at infinity. From \eqref{normal_derivative}, \eqref{order_conversion}, and \eqref{normalized_wronskian}, we have $\partial_\nu u|_{\mathrm{ext}}-\partial_\nu u|_{\mathrm{int}}=\varphi_n^m$ on $\partial\Omega_0$. By the jump relation \eqref{jump_relation}, $\Scal_{\partial\Omega_0}[\varphi_n^m]$ satisfies the same properties. Hence the difference of the two functions has no jump in the normal derivative across the smooth surface $\partial\Omega_0$, so it is harmonic in $\RR^3$ and vanishes at infinity. This proves \eqref{umn_def}, and \eqref{single_layer_trace} follows from \eqref{constantC} and \eqref{alpha_def}.
\end{proof}

From \eqref{single_layer_trace} and \eqref{normal_derivative}, the powers of $\cosh\tau_0-\cos\theta$ cancel in the inner product, and we have
\begin{equation}\label{density_gram_matrix}
\bigl\langle\varphi_n^m,\varphi_l^{m'}\bigr\rangle_{\mathscr H^*(\partial\Omega_0)}
=(2\pi)^2a^3\sinh^2\tau_0\,\kappa_{mn}^2\,b_n^{(m)}(\tau_0)\,
\delta_{mm'}\delta_{nl}.
\end{equation}
In particular, $b_n^{(m)}(\tau_0)>0$ (see also \cite[(4.4)]{AndoJiKangKawagoeMiyanishi}). The family $\{\varphi_n^m\}$ is also complete. Indeed, in the coordinates $(\phi,\theta)$, the space $H^{-1/2}(\partial\Omega_0)$ is the space $H^{-1/2}$ on the flat torus with an equivalent norm, the trigonometric polynomials are dense in it, and the multiplication by the smooth positive function $(\cosh\tau_0-\cos\theta)^{3/2}$ is an isomorphism. Hence
\begin{equation}\label{orthonormal_density_basis}
\widehat\varphi_n^m
=\frac{\varphi_n^m}
{2\pi a^{3/2}\sinh\tau_0\,\kappa_{mn}\sqrt{b_n^{(m)}(\tau_0)}},
\qquad m,n\in\ZZ,
\end{equation}
form an orthonormal basis of $\mathscr H^*(\partial\Omega_0)$.

\subsection{Fourier blocks}
\label{subsec:fourier_block_expansion}

We set
\begin{equation}\label{Dnn_formula}
d_n^m=\frac{\sinh\tau_0}{2}\,\frac{d}{d\tau_0}b_n^{(m)}(\tau_0),
\qquad
\beta_n^m=d_n^m+\frac{\sinh\tau_0}{2}\,b_n^{(m)}(\tau_0).
\end{equation}
Let $\mathscr H_m$ be the closed span of $\{\widehat\varphi_n^m\}_{n\in\ZZ}$ in $\mathscr H^*(\partial\Omega_0)$. By the rotational symmetry, each $\mathscr H_m$ is a reducing subspace of $\Kcal^*_{\partial\Omega_0}$, and $\mathscr H^*(\partial\Omega_0)=\bigoplus_{m\in\ZZ}\mathscr H_m$. Let $U_m:\ell^2(\ZZ)\to\mathscr H_m$ be the unitary operator given by $U_me_n=\widehat\varphi_n^m$, where $\{e_n\}$ is the standard basis of $\ell^2(\ZZ)$. We define the matrices
\[
\mathcal D_{\tau_0}^{(m)}=\operatorname{diag}\bigl(d_n^m\bigr),\qquad
\Acal_{\tau_0}^{(m)}=\operatorname{diag}\Bigl(\sqrt{b_n^{(m)}(\tau_0)}\Bigr),\qquad
T_{\tau_0}=\bigl(e^{-|n-l|\tau_0}\bigr)_{n,l\in\ZZ}.
\]

\begin{prop}\label{prop:Kexpansion}
For every $m\in\ZZ$, we have
\begin{equation}\label{similar_block}
\widetilde K_{\tau_0}^{(m)}
:=U_m^*\Kcal_{\partial\Omega_0}^*U_m
=\mathcal D_{\tau_0}^{(m)}
+\frac{\sinh\tau_0}{2}\,\Acal_{\tau_0}^{(m)}T_{\tau_0}\Acal_{\tau_0}^{(m)} .
\end{equation}
In particular, $\widetilde K_{\tau_0}^{(m)}$ is a compact self-adjoint operator on $\ell^2(\ZZ)$ with diagonal entries $\beta_n^m$, and the blocks for $m$ and $-m$ coincide.
\end{prop}

\begin{proof}
Fix $m$ and $l$, and let $p(\tau)=P^{-|m|}_{|l|-1/2}(\cosh\tau)$ and $q(\tau)=(-1)^{|m|}Q^{|m|}_{|l|-1/2}(\cosh\tau)$, so that $b_l^{(m)}=pq$. From \eqref{umn_def} and \eqref{order_conversion}, we have
\[
\Scal_{\partial\Omega_0}[\varphi_l^m]
=-a\sinh\tau_0\,\kappa_{ml}\sqrt{\cosh\tau-\cos\theta}\,e^{im\phi+il\theta}
\times
\begin{cases}
p(\tau_0)q(\tau),&\tau>\tau_0,\\
q(\tau_0)p(\tau),&\tau<\tau_0.
\end{cases}
\]
By \eqref{jump_relation}, $\Kcal^*_{\partial\Omega_0}[\varphi_l^m]$ is the average of the interior and the exterior normal derivatives. The derivative of the factor $\sqrt{\cosh\tau-\cos\theta}$ gives the same contribution from both sides, while that of the radial factor gives $\frac12\sinh\tau_0\,(pq)'(\tau_0)$ on average. Hence, by \eqref{normal_derivative}, we have
\begin{equation}\label{two_term_action}
\Kcal_{\partial\Omega_0}^*[\varphi_l^m]
=\Bigl(d_l^m+\frac{\sinh^2\tau_0}{2(\cosh\tau_0-\cos\theta)}\,b_l^{(m)}(\tau_0)\Bigr)\varphi_l^m .
\end{equation}
This identity corresponds to \cite[(4.1)]{AndoJiKangKawagoeMiyanishi} in our notation. The first term is diagonal. The second term depends on $\theta$ and couples different cross-sectional modes. From the Poisson kernel expansion $\sinh\tau_0/(\cosh\tau_0-\cos\theta)=\sum_{j\in\ZZ}e^{-|j|\tau_0}e^{ij\theta}$ (see, for example, \cite{CohlDominici}), we have
\[
\frac{\sinh\tau_0}{\cosh\tau_0-\cos\theta}\,\varphi_l^m
=\sum_{n\in\ZZ}e^{-|n-l|\tau_0}\frac{\kappa_{ml}}{\kappa_{mn}}\,\varphi_n^m .
\]
In terms of the basis \eqref{orthonormal_density_basis}, the coefficient of $\widehat\varphi_n^m$ in the second term of \eqref{two_term_action} applied to $\widehat\varphi^m_l$ becomes $\frac12\sinh\tau_0\sqrt{b_n^{(m)}b_l^{(m)}}\,e^{-|n-l|\tau_0}$. The series converges in $\mathscr H^*(\partial\Omega_0)$ since the weights are bounded by Lemma~\ref{lem:global_weight_bound} below. This proves \eqref{similar_block}.
\end{proof}

We will also use an expression of $d_n^m$ without the derivative. Applying the recurrence relation \cite[Eq.~14.10.4]{DLMF}, $(\eta^2-1)(F^\mu_\xi)'=(\xi-\mu+1)F^\mu_{\xi+1}-(\xi+1)\eta F^\mu_\xi$ for $F=P,Q$, to \eqref{alpha_def}, we have
\begin{align}\label{Dnn_neighbor_formula}
d_n^m
=\frac{(-1)^{|m|}}{2}\Bigl[&
(|n|-|m|+\tfrac12)\,Q^{|m|}_{|n|+\frac12}P^{-|m|}_{|n|-\frac12}
+(|n|+|m|+\tfrac12)\,Q^{|m|}_{|n|-\frac12}P^{-|m|}_{|n|+\frac12}
\nonumber\\
&-(2|n|+1)\cosh\tau_0\,Q^{|m|}_{|n|-\frac12}P^{-|m|}_{|n|-\frac12}\Bigr],
\end{align}
where all functions are evaluated at $\cosh\tau_0$.

Note that $T_{\tau_0}$ is a Toeplitz matrix. Under the Fourier series map $\ell^2(\ZZ)\to L^2(-\pi,\pi)$, $v\mapsto(2\pi)^{-1/2}\sum_nv_ne^{inp}$, it becomes the multiplication by
\begin{equation}\label{E_symbol}
\sum_{j\in\ZZ}e^{-|j|\tau_0}e^{ijp}=\frac{\sinh\tau_0}{\cosh\tau_0-\cos p},
\end{equation}
whose range is $[\tanh(\tau_0/2),\coth(\tau_0/2)]$. Hence $T_{\tau_0}$ is positive definite and $\|T_{\tau_0}\|=\coth(\tau_0/2)$.

\subsection{Uniform bound of the weights}
\label{subsec:weighted_toeplitz}

\begin{lemma}\label{lem:global_weight_bound}
For $\tau_0>0$ and $m,n\in\ZZ$, we have $0<b_n^{(m)}(\tau_0)\le b_0^{(0)}(\tau_0)$. Moreover, $b_0^{(0)}(\tau_0)\le C(1+|\log\tau_0|)$ for $0<\tau_0\le1$, where $C$ is a constant independent of $\tau_0$.
\end{lemma}

\begin{proof}
The single-layer potential of the densities $(\cosh\tau_0-\cos\theta)^{3/2}e^{im\phi+in\theta}$ was computed in \cite[(3.4)--(3.6), (4.2)--(4.4)]{AndoJiKangKawagoeMiyanishi} in terms of the coefficients $s_{m,n}(\xi)$ with $\xi=\operatorname{sech}\tau_0$. Comparing it with \eqref{single_layer_trace}, we have $b_n^{(m)}(\tau_0)=s_{m,n}(\xi)/(4\pi\sqrt2)$, that is,
\begin{equation}\label{weights_fourier_integral}
b_n^{(m)}(\tau_0)
=\frac1{4\pi\sqrt2}\int_0^{2\pi}\!\!\int_0^{2\pi}
\frac{e^{-i(m\phi+n\theta)}\,d\phi\,d\theta}
{\sqrt{1+\sinh^2\tau_0(1-\cos\phi)-\cos\theta}} .
\end{equation}
Since $b_n^{(m)}$ is a Fourier coefficient of a positive integrable function, it is bounded by the zero-frequency coefficient. The logarithmic bound follows from $b_0^{(0)}=P_{-1/2}Q_{-1/2}$ at $\cosh\tau_0$ and the behavior of these functions near $1$ (see \cite[Eqs.~14.8.7, 14.8.9, 14.3.10]{DLMF}).
\end{proof}

Here and in what follows, $C$ denotes a constant independent of $\tau_0$ and of the Fourier indices, which may change from line to line, and a subscript indicates any further dependence.

\section{Bessel asymptotics in the horn-torus limit}
\label{sec:asymptotics}

Since the interaction factor $e^{-|n-l|\tau_0}$ depends on $(n-l)\tau_0$, it is natural to use the variable $s=|n|\tau_0$. For fixed $s>0$, the degree $s/\tau_0-1/2$ of the Legendre functions tends to infinity while the argument $\cosh\tau_0$ tends to $1$, and in this regime the Legendre functions converge to the modified Bessel functions. In this section, we show that $b_n^{(m)}\to I_mK_m(s)$ and $d_n^m\to D_m(s)$, where
\begin{equation}\label{bessel_coefficients}
D_m(s)=\frac{s}{2}\frac{d}{ds}\bigl(I_m(s)K_m(s)\bigr).
\end{equation}
Since the blocks for $\pm m$ coincide, we assume $m\ge0$ in this section.

\subsection{Uniform asymptotics}
\label{subsec:bessel_asymptotics}

To compare the Legendre functions of neighboring degrees, we allow real degrees and set
\[
\mathscr P_{\tau_0}(s)=\Bigl(\frac{s}{\tau_0}\Bigr)^m
P^{-m}_{s/\tau_0-1/2}(\cosh\tau_0),
\qquad
\mathscr Q_{\tau_0}(s)=(-1)^m\Bigl(\frac{\tau_0}{s}\Bigr)^m
Q^m_{s/\tau_0-1/2}(\cosh\tau_0).
\]

\begin{lemma}\label{lem:uniform_bessel_asym}
Let $m\ge0$ and $0<\delta<R<\infty$. For $j=0,1$, we have
\begin{equation}\label{quant_P_diff}
\partial_s^j\mathscr P_{\tau_0}(s)=I_m^{(j)}(s)+O(\tau_0),
\qquad
\partial_s^j\mathscr Q_{\tau_0}(s)=K_m^{(j)}(s)+O(\tau_0),
\end{equation}
uniformly for $s\in[\delta,R]$, where the implied constants depend on $m$, $\delta$, and $R$.
\end{lemma}

\begin{proof}
The integral representations \eqref{P_integral} and \eqref{Q_integral} hold for real degrees satisfying $s/\tau_0+1/2>m$. By $\Gamma(z+1)=z\Gamma(z)$, the gamma factors in \eqref{P_integral} and \eqref{Q_integral}, multiplied by the powers of $s/\tau_0$, become
\begin{equation}\label{gamma_prefactor_products}
\prod_{j=0}^{m-1}\Bigl(1+\frac{(j+\frac12)\tau_0}{s}\Bigr)^{-1}
\quad\text{and}\quad
\prod_{j=0}^{m-1}\Bigl(1-\frac{(j+\frac12)\tau_0}{s}\Bigr),
\end{equation}
which are $1+O(\tau_0)$ together with their $s$-derivatives, uniformly on $[\delta,R]$. In both integrals, we write
\begin{equation}\label{toroidal_power_rewrite}
(\cosh\tau_0+x\sinh\tau_0)^{-s/\tau_0-1/2}
=(\cosh\tau_0+x\sinh\tau_0)^{-1/2}e^{-s\rho_{\tau_0}(x)},
\end{equation}
where $\rho_{\tau_0}(x)=\tau_0^{-1}\log(\cosh\tau_0+x\sinh\tau_0)$.

For the first kind, we have $x=\cos\vartheta\in[-1,1]$. By Taylor's theorem, $\rho_{\tau_0}(x)=x+O(\tau_0)$ and $(\cosh\tau_0+x\sinh\tau_0)^{-1/2}=1+O(\tau_0)$ uniformly in $x$. Hence the integrand and its $s$-derivative differ from $(-x)^je^{-sx}$ by $O(\tau_0)$. Since $\frac{(-1)^m}{2\pi}\int_0^{2\pi}e^{-s\cos\vartheta}\cos(m\vartheta)\,d\vartheta=I_m(s)$ (see \cite[Eq.~10.32.3]{DLMF}), we have the estimate for $\mathscr P_{\tau_0}$.

For the second kind, $x=\cosh t\ge1$ is unbounded, and we need a majorant independent of $\tau_0$. One can easily show that $1+\tau_0x\le\cosh\tau_0+x\sinh\tau_0\le e^{\tau_0x}$ for $x\ge1$, and hence
\begin{equation}\label{second_kind_rho_bounds}
\frac{\log(1+\tau_0x)}{\tau_0}\le\rho_{\tau_0}(x)\le x,
\qquad
0\le x-\rho_{\tau_0}(x)\le\frac{\tau_0x^2}{2}.
\end{equation}
We also have $0\le1-(\cosh\tau_0+x\sinh\tau_0)^{-1/2}\le C\tau_0(1+x)$ for $0<\tau_0\le1$. By the mean value theorem, it follows that, for $s\in[\delta,R]$ and $j=0,1$,
\begin{equation}\label{second_kind_integrand_error}
\Bigl|\partial_s^j\bigl[(\cosh\tau_0+x\sinh\tau_0)^{-1/2}e^{-s\rho_{\tau_0}(x)}\bigr]-(-x)^je^{-sx}\Bigr|
\le C\tau_0(1+x^{j+2})e^{-s\rho_{\tau_0}(x)}.
\end{equation}
Let $\tau_*\in(0,1]$ satisfy $\delta/\tau_*>m+3$. Since $\tau\mapsto\log(1+\tau x)/\tau$ is decreasing, \eqref{second_kind_rho_bounds} implies $e^{-s\rho_{\tau_0}(x)}\le(1+\tau_*x)^{-\delta/\tau_*}$ for $0<\tau_0\le\tau_*$. After the substitution $x=\cosh t$, the right-hand side of \eqref{second_kind_integrand_error} multiplied by $\cosh(mt)$ is bounded by $C\tau_0$ times a function of order $e^{(m+3-\delta/\tau_*)t}$, which is integrable on $(0,\infty)$. The same majorant without the factor $\tau_0$ bounds the integrands, which justifies the differentiation under the integral sign. Since $\int_0^\infty e^{-s\cosh t}\cosh(mt)\,dt=K_m(s)$ (see \cite[Eq.~10.32.9]{DLMF}), we have the estimate for $\mathscr Q_{\tau_0}$.
\end{proof}

\begin{prop}\label{prop:D_asym}
For fixed $m\ge0$, we have
\begin{equation}\label{D_asym}
b_n^{(m)}(\tau_0)=I_m(n\tau_0)K_m(n\tau_0)+O(\tau_0),
\qquad
d_n^m=D_m(n\tau_0)+O(\tau_0),
\end{equation}
uniformly for $n\tau_0\in[\delta,R]$.
\end{prop}

\begin{proof}
The first estimate follows from $b_n^{(m)}(\tau_0)=\mathscr P_{\tau_0}(n\tau_0)\mathscr Q_{\tau_0}(n\tau_0)$ and Lemma~\ref{lem:uniform_bessel_asym}. For the second one, we apply Lemma~\ref{lem:uniform_bessel_asym} on a slightly larger interval and integrate the derivative estimate over $[s,s+\tau_0]$. Since $I_m$ and $K_m$ are positive, we have
\[
\frac{P^{-m}_{n+1/2}}{P^{-m}_{n-1/2}}
=1+\tau_0\Bigl(-\frac{m}{n\tau_0}+\frac{I_m'}{I_m}\Bigr)+O(\tau_0^2),
\qquad
\frac{Q^m_{n+1/2}}{Q^m_{n-1/2}}
=1+\tau_0\Bigl(\frac{m}{n\tau_0}+\frac{K_m'}{K_m}\Bigr)+O(\tau_0^2),
\]
where the Bessel functions are evaluated at $n\tau_0$. We divide \eqref{Dnn_neighbor_formula} by $b_n^{(m)}=(-1)^mP^{-m}_{n-1/2}Q^m_{n-1/2}$ and substitute these ratios. Then the terms of order $n$ cancel, the terms involving $\pm m/n$ give $-2m^2/n=O(\tau_0)$, and $\cosh\tau_0=1+O(\tau_0^2)$. Hence
\[
d_n^m=\frac{n\tau_0}{2}\,I_mK_m(n\tau_0)\Bigl(\frac{I_m'}{I_m}+\frac{K_m'}{K_m}\Bigr)(n\tau_0)+O(\tau_0)
=D_m(n\tau_0)+O(\tau_0). \qedhere
\]
\end{proof}

\subsection{Endpoint behavior}
\label{subsec:endpoint_behavior}

From the standard expansions \cite[Eqs.~10.25.2, 10.31.1--10.31.2, 10.40.1--10.40.4, 10.40.6]{DLMF}, we have, for fixed $m$,
\begin{equation}\label{endpoint_infinity}
I_m(s)K_m(s)=\frac1{2s}+O(s^{-3}),\qquad
D_m(s)=-\frac1{4s}+O(s^{-3})\qquad\text{as }s\to\infty,
\end{equation}
where the terms of order $s^{-2}$ cancel. As $s\downarrow0$, we have $I_mK_m(s)\to1/(2m)$ and $D_m(s)\to0$ for $m\ge1$, and
\begin{equation}\label{endpoint_zero}
I_0(s)K_0(s)=-\log\frac s2-\gamma+O(s^2|\log s|),
\qquad
D_0(s)=-\frac12+O(s^2|\log s|),
\end{equation}
where $\gamma$ is Euler's constant.

\section{Spectral limit for the horn torus}
\label{sec:spectral_consequences}

\subsection{Uniform bounds}
\label{subsec:positive_count}

The fixed-$m$ asymptotics in section~\ref{sec:asymptotics} give no information on the blocks with large $m$. In this subsection, we derive bounds that are uniform in both Fourier indices and show that only finitely many blocks have eigenvalues above a fixed positive level.

\begin{lemma}\label{lem:diagonal_range}
For every $m\ge0$, we have $D_m(s)<0$ for $s>0$, so that $I_mK_m$ is strictly decreasing. Moreover, $D_0((0,\infty))=(-1/2,0)$.
\end{lemma}

\begin{proof}
For $0<u<v$ and $s>0$, it follows from \cite[Eq.~10.22.69]{DLMF} with $z=is$ and \cite[Eqs.~10.27.6, 10.27.8]{DLMF} that $\int_0^\infty tJ_m(ut)J_m(vt)(t^2+s^2)^{-1}dt=I_m(us)K_m(vs)$. Since the integrand is $O(t^{2m+1})$ near $0$ and $O(t^{-2})$ at infinity, we can let $u,v\to1$ by the dominated convergence theorem. By differentiating in $s$, we have
\begin{equation}\label{bessel_diagonal_negative_integral}
I_m(s)K_m(s)=\int_0^\infty\frac{tJ_m(t)^2}{t^2+s^2}\,dt,
\qquad
D_m(s)=-s^2\int_0^\infty\frac{tJ_m(t)^2}{(t^2+s^2)^2}\,dt<0.
\end{equation}
For $m=0$, the relations $I_0'=I_1$, $K_0'=-K_1$, and $I_0K_1+I_1K_0=1/s$ (see \cite[Eqs.~10.29.3, 10.28.2]{DLMF}) give
\begin{equation}\label{zero_order_diagonal_lower_bound}
D_0(s)=sI_1(s)K_0(s)-\frac12>-\frac12 .
\end{equation}
Since $D_0$ is continuous and tends to $-1/2$ and $0$ at the endpoints by \eqref{endpoint_infinity} and \eqref{endpoint_zero}, its range is $(-1/2,0)$.
\end{proof}

\begin{lemma}\label{lem:green_weight_bounds}
Let $\iota(t)=-\log\tanh(t/2)$ for $t>0$. For $0<\tau_0\le1$, we have
\begin{align}
 &b_n^{(m)}(\tau_0)\le\frac{C}{|m|},
 &&m\ne0,\ n\in\ZZ,\label{weight_m_decay}\\
 &b_n^{(m)}(\tau_0)\le\frac{1}{2|n|\sinh\tau_0},
 &&m\ne0,\ n\ne0,\label{weight_n_decay}\\
 &b_n^{(0)}(\tau_0)\le\frac{\tau_0}{\sinh\tau_0}\,I_0K_0(|n|\tau_0),
 &&n\ne0.\label{weight_zero_mode_decay}
\end{align}
For the diagonal coefficients, we have $d_n^m\le0$ for $n\ne0$ and every $\tau_0>0$, $d_0^0<0$, and
\begin{equation}\label{positive_diagonal_rank_one}
d_0^m\le\frac{C}{m^2\,\iota(\tau_0)^2}\le\frac{C}{m^2},\qquad m\ne0,\ 0<\tau_0\le1.
\end{equation}
Consequently,
\begin{equation}\label{positive_diagonal_vanish}
\sup_{m,n\in\ZZ}d_n^m\le\frac{C}{\iota(\tau_0)^2},\qquad 0<\tau_0\le1.
\end{equation}
\end{lemma}

The proof of Lemma~\ref{lem:green_weight_bounds} is given in Appendix~\ref{app:weight_bounds}. Note that the right-hand side of \eqref{positive_diagonal_vanish} tends to $0$ as $\tau_0\downarrow0$, since $\iota(\tau_0)\sim\log(2/\tau_0)\to\infty$.

We denote the interaction part of \eqref{similar_block} by
\[
B_{\tau_0}^{(m)}=\frac{\sinh\tau_0}{2}\,\Acal_{\tau_0}^{(m)}T_{\tau_0}\Acal_{\tau_0}^{(m)}.
\]
Since $T_{\tau_0}$ is positive, so is $B_{\tau_0}^{(m)}$. From \eqref{weight_n_decay}, \eqref{weight_zero_mode_decay}, and \eqref{endpoint_infinity}, we have $b_n^{(m)}(\tau_0)\to0$ as $|n|\to\infty$, and hence $\Acal_{\tau_0}^{(m)}$ and $B_{\tau_0}^{(m)}$ are compact.

\begin{theorem}\label{thm:positive_count}
For $0<\tau_0\le1$, we have
\begin{equation}\label{uniform_interaction_bounds}
\bigl\|B_{\tau_0}^{(m)}\bigr\|_{\mathrm{HS}}^2\le\frac C{1+|m|},
\qquad
\bigl\|B_{\tau_0}^{(m)}\bigr\|\le\frac C{|m|},
\qquad
\sup\sigma\bigl(\widetilde K_{\tau_0}^{(m)}\bigr)\le\frac C{|m|}\quad(m\ne0),
\end{equation}
where $\|\cdot\|_{\mathrm{HS}}$ denotes the Hilbert--Schmidt norm. Moreover, for every $0<\varepsilon<1/2$, the number $N_+(\varepsilon,\Kcal_{\partial\Omega_0}^*)$ of eigenvalues of $\Kcal_{\partial\Omega_0}^*$ greater than $\varepsilon$, counted with multiplicity, satisfies
\begin{equation}\label{positive_count_bound}
N_+\bigl(\varepsilon,\Kcal_{\partial\Omega_0}^*\bigr)\le C\varepsilon^{-2}\log(1/\varepsilon).
\end{equation}
\end{theorem}

\begin{proof}
For $m,n\ne0$, it follows from \eqref{weight_m_decay} and \eqref{weight_n_decay} that $[b_n^{(m)}]^2\le C\min\{m^{-2},(\tau_0n)^{-2}\}\le C/(m^2+\tau_0^2n^2)$, and this bound also holds for $n=0$. By comparing the sum with an integral, we have
\[
\sum_{n\in\ZZ}\bigl[b_n^{(m)}(\tau_0)\bigr]^2\le\frac{C}{m^2}+C\int_0^\infty\frac{dx}{m^2+\tau_0^2x^2}\le\frac{C}{|m|\tau_0}\qquad(m\ne0).
\]
For $m=0$, from \eqref{weight_zero_mode_decay}, the monotonicity of $I_0K_0$, and Lemma~\ref{lem:global_weight_bound}, we have
\[
\sum_{n\in\ZZ}\bigl[b_n^{(0)}(\tau_0)\bigr]^2\le C(1+|\log\tau_0|)^2+\frac{2}{\tau_0}\int_0^\infty\bigl[I_0K_0(s)\bigr]^2ds\le\frac{C}{\tau_0},
\]
where the integral is finite by \eqref{endpoint_infinity} and \eqref{endpoint_zero}. Therefore,
\begin{equation}\label{weight_square_sum}
\sum_{n\in\ZZ}\bigl[b_n^{(m)}(\tau_0)\bigr]^2\le\frac{C}{(1+|m|)\tau_0}.
\end{equation}
By grouping the entries of $B^{(m)}_{\tau_0}$ according to $j=n-l$ and applying the Cauchy--Schwarz inequality, we have
\[
\bigl\|B_{\tau_0}^{(m)}\bigr\|_{\mathrm{HS}}^2
=\frac{\sinh^2\tau_0}{4}\sum_{j}e^{-2|j|\tau_0}\sum_{n}b_n^{(m)}b_{n-j}^{(m)}
\le\frac{\sinh\tau_0\cosh\tau_0}{4}\sum_n\bigl[b_n^{(m)}\bigr]^2
\le\frac{C}{1+|m|},
\]
where we used $\sum_je^{-2|j|\tau_0}=\coth\tau_0$. The bound $\|B_{\tau_0}^{(m)}\|\le\frac12\sinh\tau_0\coth(\tau_0/2)\sup_nb_n^{(m)}\le C/|m|$ follows from \eqref{weight_m_decay}. Since $d_n^m\le0$ for $n\ne0$ and $d_0^m\le C/m^2$, we have $\langle\widetilde K^{(m)}_{\tau_0}u,u\rangle\le Cm^{-2}|u_0|^2+\|B_{\tau_0}^{(m)}\|\le C/|m|$ for every unit vector $u$, which proves \eqref{uniform_interaction_bounds}.

We now prove \eqref{positive_count_bound}. Fix $m$, and let $u_1,\dots,u_N$ be orthonormal eigenvectors of $\widetilde K^{(m)}_{\tau_0}$ with eigenvalues $\lambda_j>\varepsilon$. From the sign of the diagonal entries, we have
\[
\varepsilon<\lambda_j\le\bigl\|B_{\tau_0}^{(m)}u_j\bigr\|+\frac{C}{(1+|m|)^2}|(u_j)_0|^2 .
\]
By squaring both sides and summing over $j$, and using $\sum_j\|Bu_j\|^2\le\|B\|_{\mathrm{HS}}^2$ and $\sum_j|(u_j)_0|^2\le1$, we have $N\varepsilon^2\le C/(1+|m|)$. By \eqref{uniform_interaction_bounds}, the blocks with $|m|>C_1/\varepsilon$ have no eigenvalues above $\varepsilon$ for some constant $C_1$. Hence
\[
N_+\bigl(\varepsilon,\Kcal^*_{\partial\Omega_0}\bigr)
=\sum_{|m|\le C_1/\varepsilon}N_+\bigl(\varepsilon,\widetilde K^{(m)}_{\tau_0}\bigr)
\le\frac{C}{\varepsilon^2}\sum_{0\le m\le C_1/\varepsilon}\frac1{1+m}
\le C\varepsilon^{-2}\log(1/\varepsilon). \qedhere
\]
\end{proof}

\subsection{Limit of a single Fourier block}
\label{subsec:continuum_operator}

To find the limit of the discrete blocks, we identify $\ell^2(\ZZ)$ with the subspace of $L^2(\RR)$ consisting of the functions that are constant on the cells $E_{n,\tau_0}=[(n-\frac12)\tau_0,(n+\frac12)\tau_0)$. Let $b_{\tau_0,m}$ and $d_{\tau_0,m}$ be the step functions equal to $b_n^{(m)}(\tau_0)$ and $d_n^m$ on $E_{n,\tau_0}$, respectively. We denote by $M_f$ the multiplication by $f$, and by $\mathcal V_{\tau_0,m}$ the integral operator with the kernel
\[
v_{\tau_0,m}(s,t)=\frac{\sinh\tau_0}{2\tau_0}\sqrt{b_{\tau_0,m}(s)\,b_{\tau_0,m}(t)}\;e^{-\tau_0|n-l|},
\qquad s\in E_{n,\tau_0},\ t\in E_{l,\tau_0}.
\]
On the cellwise constant functions, the operator $M_{d_{\tau_0,m}}+\mathcal V_{\tau_0,m}$ coincides with $\widetilde K^{(m)}_{\tau_0}$ in the orthonormal basis $\{\tau_0^{-1/2}\mathbf 1_{E_{n,\tau_0}}\}$, as we show in \eqref{continuum_kernel_matrix_action} below. We define the operator $\Lcal_m$ on $L^2(\RR)$ by $\Lcal_m=M_{D_m(|\cdot|)}+\mathcal V_m$, where $D_m$ is given by \eqref{bessel_coefficients} and $\mathcal V_m$ is the integral operator with the kernel
\begin{equation}\label{continuum_kernel}
v_m(s,t)=\frac12\sqrt{I_mK_m(|s|)}\;e^{-|s-t|}\sqrt{I_mK_m(|t|)} .
\end{equation}

\begin{prop}\label{prop:continuum_convergence}
Let $m\ge0$. The operator $\mathcal V_m$ is positive and Hilbert--Schmidt, and as $\tau_0\downarrow0$,
\[
\|\mathcal V_{\tau_0,m}-\mathcal V_m\|_{\mathrm{HS}}\to0,
\qquad
M_{d_{\tau_0,m}}+\mathcal V_{\tau_0,m}\to\Lcal_m\ \text{ strongly}.
\]
The operator $\Lcal_m$ is self-adjoint with $\|\Lcal_m\|\le1/2$, and
\[
\sigma_{\mathrm{ess}}(\Lcal_m)=\overline{\{D_m(s):s>0\}}\subset[-1/2,0],
\qquad
\sigma_{\mathrm{ess}}(\Lcal_0)=[-1/2,0].
\]
The positive spectrum of $\Lcal_m$ is nonempty and consists of eigenvalues of finite multiplicity, which can accumulate only at $0$. Moreover, $\sup\sigma(\Lcal_m)\le1/(2m)$ for $m\ge1$.
\end{prop}

\begin{proof}
For $s\in E_{n,\tau_0}$, we have $\bigl||n|\tau_0-|s|\bigr|\le\tau_0/2$. Hence Proposition~\ref{prop:D_asym} gives $b_{\tau_0,m}(s)\to I_mK_m(|s|)$ and $d_{\tau_0,m}(s)\to D_m(|s|)$ for $s\ne0$. We now find a majorant for $0<\tau_0\le1$. If $m=0$, from \eqref{weight_zero_mode_decay}, the monotonicity of $I_0K_0$, and $|s|\le\frac32|n|\tau_0$ for $n\neq0$, we have $b_{\tau_0,0}(s)\le I_0K_0(2|s|/3)$. On the central cell, Lemma~\ref{lem:global_weight_bound} gives $b_0^{(0)}(\tau_0)\le C(1+|\log|s||)$. If $m\ge1$, we have $b_{\tau_0,m}(s)\le C/m$ by \eqref{weight_m_decay}, and $b_{\tau_0,m}(s)\le3/(4|s|)$ for $|s|>1$ by \eqref{weight_n_decay}. Therefore,
\[
0<b_{\tau_0,m}(s)\le F_m(s):=C_m\begin{cases}1+|\log|s||,&|s|\le1,\\|s|^{-1},&|s|>1,\end{cases}
\]
and $F_m\in L^2(\RR)$. Since $\bigl|\tau_0|n-l|-|s-t|\bigr|\le\tau_0$, we have $v_{\tau_0,m}\to v_m$ for $s,t\ne0$ and
\begin{equation}\label{kernel_majorant}
|v_{\tau_0,m}(s,t)|^2+|v_m(s,t)|^2\le C\,F_m(s)F_m(t)\,e^{-2|s-t|}.
\end{equation}
By the Cauchy--Schwarz inequality with the substitution $h=s-t$, the integral of the right-hand side over $\RR^2$ is bounded by $C\|F_m\|_{L^2}^2$. By the dominated convergence theorem, $\mathcal V_m$ is Hilbert--Schmidt and $\|\mathcal V_{\tau_0,m}-\mathcal V_m\|_{\mathrm{HS}}\to0$. Since $v_m$ is real and symmetric, $\mathcal V_m$ is self-adjoint.

Since $\widetilde K_{\tau_0}^{(m)}$ is a restriction of the NP operator, we have $\|\widetilde K^{(m)}_{\tau_0}\|\le\frac12$, and hence $|\beta_n^m|\le\frac12$. By Lemma~\ref{lem:global_weight_bound}, we have
\begin{equation}\label{continuum_diagonal_bound}
|d_n^m|\le\frac12+\frac{\sinh\tau_0}{2}b_n^{(m)}(\tau_0)\le\frac12+C\tau_0(1+|\log\tau_0|).
\end{equation}
Hence $M_{d_{\tau_0,m}}\to M_{D_m(|\cdot|)}$ strongly by the dominated convergence theorem, and together with the Hilbert--Schmidt convergence, $M_{d_{\tau_0,m}}+\mathcal V_{\tau_0,m}\to\Lcal_m$ strongly.

For $f\in L^2(\RR)$, let $u_n=\tau_0^{-1/2}\int_{E_{n,\tau_0}}f$ and $r=f-\tau_0^{-1/2}\sum_nu_n\mathbf 1_{E_{n,\tau_0}}$. Then $r$ has zero mean on every cell and $\|f\|^2=\|u\|_{\ell^2}^2+\|r\|^2$. Since $v_{\tau_0,m}$ is constant on each product of cells, we have, for $s\in E_{n,\tau_0}$,
\begin{equation}\label{continuum_kernel_matrix_action}
(\mathcal V_{\tau_0,m}f)(s)=\tau_0^{-1/2}\bigl(B^{(m)}_{\tau_0}u\bigr)_n,
\qquad
\bigl((M_{d_{\tau_0,m}}+\mathcal V_{\tau_0,m})f\bigr)(s)=\tau_0^{-1/2}\bigl(\widetilde K^{(m)}_{\tau_0}u\bigr)_n+d_n^m\,r(s).
\end{equation}
From the first identity, we have $\langle\mathcal V_{\tau_0,m}f,f\rangle=\langle B^{(m)}_{\tau_0}u,u\rangle\ge0$, and the positivity of $\mathcal V_m$ follows by taking the limit. Since the two terms in the second identity are orthogonal, we have
\[
\bigl\|(M_{d_{\tau_0,m}}+\mathcal V_{\tau_0,m})f\bigr\|^2\le\frac14\|u\|^2+\Bigl(\frac12+C\tau_0(1+|\log\tau_0|)\Bigr)^2\|r\|^2,
\]
and by letting $\tau_0\downarrow0$, we have $\|\Lcal_m\|\le\frac12$.

The spectrum of the multiplication operator $M_{D_m(|\cdot|)}$ is the essential range $\overline{\{D_m(s):s>0\}}$ of the continuous function $D_m$, and it consists only of essential spectrum, since each level set either has measure zero or supports an infinite-dimensional eigenspace. Since $\mathcal V_m$ is compact, Weyl's theorem \cite[Theorem~6.19]{Teschl2014} gives $\sigma_{\mathrm{ess}}(\Lcal_m)=\overline{\{D_m(s):s>0\}}$. This set is contained in $[-1/2,0]$ by Lemma~\ref{lem:diagonal_range} and $\|\Lcal_m\|\le\frac12$, and it is equal to $[-1/2,0]$ for $m=0$. Hence the positive spectral points are isolated eigenvalues of finite multiplicity, which can accumulate only at $0$.

To show that the positive spectrum is nonempty, let $f_R=2^{-1/2}\mathbf 1_{[R,R+2]}$. By \eqref{endpoint_infinity}, $RD_m(R+x)\to-\frac14$ and $RI_mK_m(R+x)\to\frac12$ uniformly for $x\in[0,2]$. Hence
\[
\lim_{R\to\infty}R\,\langle\Lcal_mf_R,f_R\rangle
=-\frac14+\frac18\int_0^2\!\!\int_0^2e^{-|x-y|}\,dy\,dx
=-\frac14+\frac{2+2e^{-2}}{8}=\frac{e^{-2}}{4}>0 ,
\]
so that $\sup\sigma(\Lcal_m)>0$. Since the essential spectrum is nonpositive, $\sup\sigma(\Lcal_m)$ is an eigenvalue.

Finally, let $m\ge1$. By Lemma~\ref{lem:diagonal_range} and subsection~\ref{subsec:endpoint_behavior}, we have $0<I_mK_m\le1/(2m)$. Since the convolution with $\frac12e^{-|s|}$ has norm at most $1$ by Young's inequality, we have $\|\mathcal V_m\|\le1/(2m)$. Since $D_m\le0$, the quadratic form of $\Lcal_m$ is bounded by that of $\mathcal V_m$, and hence $\sup\sigma(\Lcal_m)\le1/(2m)$.
\end{proof}

Since the strong convergence alone does not give the convergence of eigenvalues, we use the Birman--Schwinger principle \cite{Birman1961,Schwinger1961}, which moves the multiplication part into weights and allows us to work with compact operators.

\begin{prop}\label{prop:positive_convergence}
Let $m\ge0$ and $\lambda>0$ with $\lambda\notin\sigma(\Lcal_m)$. Then, for sufficiently small $\tau_0$, we have $\lambda\notin\sigma(\widetilde K^{(m)}_{\tau_0})$ and
\[
N_+\bigl(\lambda,\widetilde K_{\tau_0}^{(m)}\bigr)=N_+(\lambda,\Lcal_m),
\]
where $N_+(\lambda,T)$ denotes the number of eigenvalues of $T$ greater than $\lambda$, counted with multiplicity. As a consequence, the positive eigenvalues of $\widetilde K^{(m)}_{\tau_0}$ converge to those of $\Lcal_m$ with multiplicity, uniformly on compact subsets of $(0,\infty)$.
\end{prop}

\begin{proof}
By \eqref{positive_diagonal_vanish}, $\sup_nd_n^m<\lambda/2$ for small $\tau_0$. Since $D_m<0$, the weights $w_{\tau_0}=(\lambda-d_{\tau_0,m})^{-1/2}$ and $w=(\lambda-D_m(|\cdot|))^{-1/2}$ are bounded above by $\sqrt{2/\lambda}$ and bounded below by a positive constant by \eqref{continuum_diagonal_bound}, and $w_{\tau_0}\to w$ almost everywhere. We set
\[
\mathcal C_{\tau_0,m}=M_{w_{\tau_0}}\mathcal V_{\tau_0,m}M_{w_{\tau_0}},
\qquad
\mathcal C_m=M_w\mathcal V_mM_w .
\]
These operators are compact and positive. Since their kernels converge almost everywhere and are dominated by the integrable function in \eqref{kernel_majorant}, we have $\|\mathcal C_{\tau_0,m}-\mathcal C_m\|_{\mathrm{HS}}\to0$.

We use Glazman's lemma \cite{BirmanSolomyak1987}: for a bounded self-adjoint operator $T$ whose spectrum above $\lambda$ consists of finitely many eigenvalues, $N_+(\lambda,T)$ is the largest dimension of a subspace on which $\langle(T-\lambda)u,u\rangle>0$ for all $u\neq0$. This number is invariant under bounded invertible changes of variables. Since $\Lcal_m-\lambda=M_{w}^{-1}(\mathcal C_m-\Id)M_{w}^{-1}$, we have $N_+(\lambda,\Lcal_m)=N_+(1,\mathcal C_m)$. For the discrete block, let $u\in\ell^2(\ZZ)$ and $f=\tau_0^{-1/2}\sum_n(\lambda-d_n^m)^{1/2}u_n\mathbf 1_{E_{n,\tau_0}}$. From \eqref{similar_block} and \eqref{continuum_kernel_matrix_action}, we have
\[
\langle(\mathcal C_{\tau_0,m}-\Id)f,f\rangle
=\langle B^{(m)}_{\tau_0}u,u\rangle-\sum_n(\lambda-d_n^m)|u_n|^2
=\bigl\langle(\widetilde K^{(m)}_{\tau_0}-\lambda)u,u\bigr\rangle .
\]
Furthermore, $\mathcal C_{\tau_0,m}$ vanishes on the orthogonal complement of the cellwise constant functions, since $\mathcal V_{\tau_0,m}$ annihilates the functions with zero mean on every cell and $M_{w_{\tau_0}}$ preserves this property. Hence $N_+(\lambda,\widetilde K^{(m)}_{\tau_0})=N_+(1,\mathcal C_{\tau_0,m})$.

It remains to compare $\mathcal C_{\tau_0,m}$ with $\mathcal C_m$. Since $\lambda\notin\sigma(\Lcal_m)$, we have $1\notin\sigma(\mathcal C_m)$. Let $\delta_\lambda=\operatorname{dist}(1,\sigma(\mathcal C_m))>0$, and let $\mathcal E$ be the spectral subspace of $\mathcal C_m$ for $(1,\infty)$, which is finite-dimensional. Then the quadratic form of $\mathcal C_m-\Id$ is bounded below by $\delta_\lambda$ on $\mathcal E$ and bounded above by $-\delta_\lambda$ on $\mathcal E^\perp$. If $\|\mathcal C_{\tau_0,m}-\mathcal C_m\|<\delta_\lambda/2$, then the quadratic form of $\mathcal C_{\tau_0,m}-\Id$ is positive on $\mathcal E\setminus\{0\}$ and negative on $\mathcal E^\perp\setminus\{0\}$. Hence $N_+(1,\mathcal C_{\tau_0,m})=\dim\mathcal E$ and $1\notin\sigma(\mathcal C_{\tau_0,m})$, which proves the first assertion. The second assertion follows by applying the first one at points $\lambda_\pm\notin\sigma(\Lcal_m)$ near each positive eigenvalue of $\Lcal_m$ and on the gaps between them.
\end{proof}

\subsection{Spectral limit}
\label{subsec:positive_convergence}

\begin{theorem}\label{thm:horn_full_limit}
Let $\Sigma_+=\bigcup_{m\ge0}\bigl(\sigma(\Lcal_m)\cap(0,1/2]\bigr)$. Then $\Sigma_+$ is an infinite set containing $1/2$ whose only accumulation point is $0$. As $\tau_0\downarrow0$,
\begin{equation}\label{horn_full_hausdorff}
d_{\mathrm H}\Bigl(\sigma\bigl(\Kcal_{\partial\Omega_0}^*\bigr),[-1/2,0]\cup\Sigma_+\Bigr)\longrightarrow0,
\end{equation}
where $d_{\mathrm H}$ denotes the Hausdorff distance between nonempty compact subsets of $\RR$, and, in particular, $\min\sigma(\Kcal^*_{\partial\Omega_0})\to-1/2$. For every $\lambda\in(0,1/2)\setminus\Sigma_+$ and sufficiently small $\tau_0$, we have
\begin{equation}\label{full_positive_count_limit}
N_+\bigl(\lambda,\Kcal_{\partial\Omega_0}^*\bigr)=N_+(\lambda,\Lcal_0)+2\sum_{m=1}^\infty N_+(\lambda,\Lcal_m),
\end{equation}
where the sum has finitely many nonzero terms. Moreover, every compact interval in $(0,1/2]\setminus\Sigma_+$ does not intersect $\sigma(\Kcal^*_{\partial\Omega_0})$ for sufficiently small $\tau_0$.
\end{theorem}

\begin{proof}
We first show that every point of $[-1/2,0]$ is approximated by the spectrum. Fix $s_0>0$ and let $n=\lfloor s_0/\tau_0\rfloor$. By Proposition~\ref{prop:D_asym} and the mean value theorem, we have $b_n^{(0)}(\tau_0)\le C_{s_0}$ and $|d_n^0-D_0(s_0)|\le C_{s_0}\tau_0$. From \eqref{similar_block}, the Cauchy--Schwarz inequality, and \eqref{weight_square_sum}, we have
\[
\bigl\|B_{\tau_0}^{(0)}e_n\bigr\|^2
=\frac{\sinh^2\tau_0}{4}\,b_n^{(0)}\sum_lb_l^{(0)}e^{-2\tau_0|l-n|}
\le\frac{\sinh^2\tau_0}{4}\,b_n^{(0)}\Bigl(\sum_l\bigl[b_l^{(0)}\bigr]^2\Bigr)^{1/2}\coth(2\tau_0)^{1/2}
\le C_{s_0}\tau_0 .
\]
Hence $\|(\widetilde K^{(0)}_{\tau_0}-D_0(s_0))e_n\|\le C_{s_0}\sqrt{\tau_0}$. Since the zeroth block is a part of the NP operator, we have $\operatorname{dist}(D_0(s_0),\sigma(\Kcal^*_{\partial\Omega_0}))\le C_{s_0}\sqrt{\tau_0}$. Since $D_0((0,\infty))$ is dense in $[-1/2,0]$ and $\operatorname{dist}(\cdot,\sigma)$ is $1$-Lipschitz, we have
\begin{equation}\label{negative_inclusion_final}
\sup_{\lambda\in[-1/2,0]}\operatorname{dist}\bigl(\lambda,\sigma(\Kcal^*_{\partial\Omega_0})\bigr)\to0 .
\end{equation}
In particular, $\min\sigma(\Kcal^*_{\partial\Omega_0})\to-1/2$.

We next consider the structure of $\Sigma_+$. By Proposition~\ref{prop:continuum_convergence}, the indices $m>1/(2\varepsilon)$ do not contribute to $\Sigma_+\cap[\varepsilon,1/2]$, and each remaining $m$ contributes finitely many points. Hence $\Sigma_+\cap[\varepsilon,1/2]$ is finite. Since $\Lcal_m$ has an eigenvalue in $(0,1/(2m)]$ for each $m\ge1$, the set $\Sigma_+$ is infinite and accumulates at $0$, and $[-1/2,0]\cup\Sigma_+$ is compact.

Let $\lambda\in(0,1/2)\setminus\Sigma_+$. Then $\lambda\notin\sigma(\Lcal_m)$ for every $m$. We choose $M$ such that $C/M<\lambda$ for the constant $C$ in \eqref{uniform_interaction_bounds} and $1/(2M)<\lambda$. Then the blocks and the limiting operators with $|m|>M$ have no spectrum above $\lambda$. By applying Proposition~\ref{prop:positive_convergence} to $0\le m\le M$, and using that the blocks for $\pm m$ coincide and act on orthogonal subspaces, we have \eqref{full_positive_count_limit}. By applying this at two points $\lambda_\pm\notin\Sigma_+$ near $\mu\in\Sigma_+$, the number of eigenvalues of the torus in $(\lambda_-,\lambda_+]$ is equal to $\dim\ker(\Lcal_0-\mu)+2\sum_{m\ge1}\dim\ker(\Lcal_m-\mu)$ for small $\tau_0$. For $\lambda_+\ge\frac12$, both counts above $\lambda_+$ vanish by the norm bounds, so the case $\mu=\frac12$ is also included.

We now prove \eqref{horn_full_hausdorff}. Every point of the limit set is approximated by the spectrum: the points in $[-1/2,0]$ by \eqref{negative_inclusion_final}, those in the finite set $\Sigma_+\cap[\varepsilon,1/2]$ by the counting result above, and those in $\Sigma_+\cap(0,\varepsilon)$ trivially since $0\in\sigma(\Kcal^*_{\partial\Omega_0})$. Conversely, suppose that there exist $\varepsilon>0$, $\tau_j\downarrow0$, and $\lambda_j\in\sigma(\Kcal^*_{\partial\Omega_0})$ for $\tau_0=\tau_j$ whose distance from $[-1/2,0]\cup\Sigma_+$ is at least $\varepsilon$. Then $\lambda_j\in[\varepsilon,1/2]$ is an eigenvalue of a block whose index is bounded by \eqref{uniform_interaction_bounds}. By passing to a subsequence, we may assume that the block index $m$ is fixed and $\lambda_j\to\lambda_*\in[\varepsilon,1/2]$. If $\lambda_*\notin\sigma(\Lcal_m)$, then Proposition~\ref{prop:positive_convergence}, applied at the endpoints of a small interval around $\lambda_*$ disjoint from $\sigma(\Lcal_m)$, shows that the block has no eigenvalues in this interval for large $j$, which is a contradiction. Hence $\lambda_*\in\Sigma_+$, which contradicts the choice of $\lambda_j$. This proves \eqref{horn_full_hausdorff}. Since $\frac12$ is an eigenvalue of the NP operator on every torus, with the equilibrium density as an eigenfunction, it belongs to the limit set and hence to $\Sigma_+$. Finally, a compact interval in $(0,1/2]\setminus\Sigma_+$ has a positive distance from the limit set, and hence it does not intersect the spectrum for small $\tau_0$.
\end{proof}

\section{Spectral limit for the thin ring}
\label{sec:thin_ring}

We now consider the limit $\tau_0\to\infty$. We set $t_0=\iota(\tau_0)$, where $\iota$ is as in Lemma~\ref{lem:green_weight_bounds}, so that $t_0\downarrow0$. Note that $\iota(t_0)=\tau_0$, since $\iota\circ\iota$ is the identity. The following lemma shows that the matrix entries for $\tau_0$ can be expressed by those for $t_0$, for which the estimates in sections~\ref{sec:asymptotics} and \ref{sec:spectral_consequences} are available. We note that the roles of the two Fourier indices are exchanged.

\begin{lemma}\label{lem:toroidal_duality}
For all $t>0$ and $m,n\in\ZZ$, we have
\begin{align}
 &b_n^{(m)}(t)=\sinh\iota(t)\;b_m^{(n)}(\iota(t)),
 \label{weight_duality}\\
 &\beta_n^m(t)=-d_m^n(\iota(t))-\frac{\cosh\iota(t)-1}{2}\,b_m^{(n)}(\iota(t)).
 \label{diagonal_duality}
\end{align}
\end{lemma}

\begin{proof}
Let $y=\iota(t)$ and $m,n\ge0$. Whipple's formulas \cite[Eqs.~14.9.16--14.9.17]{DLMF} are stated in terms of $\boldsymbol Q^\mu_\nu=e^{-\mu\pi i}Q^\mu_\nu/\Gamma(\nu+\mu+1)$ (see \cite[Eq.~14.3.10]{DLMF}). In our normalization, by using $P^{-n}_{-m-1/2}=P^{-n}_{m-1/2}$ (see \cite[Eq.~14.9.11]{DLMF}), they become
\[
P_{n-1/2}^{-m}(\cosh t)
=\sqrt{\frac{2}{\pi\sinh t}}\,
\frac{(-1)^nQ_{m-1/2}^{n}(\cosh y)}{\Gamma(m+n+\frac12)},
\]
\[
(-1)^mQ_{n-1/2}^{m}(\cosh t)
=\sqrt{\frac{\pi}{2\sinh t}}\,
\Gamma(m+n+\tfrac12)\,P_{m-1/2}^{-n}(\cosh y).
\]
By multiplying these two formulas and using $\sinh t\sinh y=1$, which follows from the definition of $\iota$, we have \eqref{weight_duality}, since the weights depend only on $|m|$ and $|n|$. By differentiating \eqref{weight_duality} with respect to $t$ and using $dy/dt=-1/\sinh t$, we have $d_n^m(t)=-\frac12\cosh y\,b_m^{(n)}(y)-d_m^n(y)$. Since $\frac12\sinh t\,b_n^{(m)}(t)=\frac12b_m^{(n)}(y)$, we have \eqref{diagonal_duality}.
\end{proof}

\begin{theorem}\label{thm:slender_ring_limit}
As $\tau_0\to\infty$, we have
\begin{equation}\label{slender_ring_hausdorff}
d_{\mathrm H}\Bigl(\sigma\bigl(\Kcal^*_{\partial\Omega_0}\bigr),[0,1/2]\Bigr)\longrightarrow0.
\end{equation}
Moreover, for sufficiently large $\tau_0$,
\begin{equation}\label{slender_ring_negative_bound}
-\frac{C}{\tau_0^2}\le\min\sigma\bigl(\Kcal^*_{\partial\Omega_0}\bigr)\le0 .
\end{equation}
\end{theorem}

\begin{proof}
Since $t_0=\log\frac{1+e^{-\tau_0}}{1-e^{-\tau_0}}$, we have $2e^{-\tau_0}\le t_0\le2e^{-\tau_0}/(1-e^{-2\tau_0})$. In particular, $t_0\le1$ for $\tau_0\ge2$, and $\sinh t_0\le Ce^{-\tau_0}$ and $1+|\log t_0|\le C\tau_0$.

We first estimate the off-diagonal part. From Lemma~\ref{lem:global_weight_bound} and \eqref{weight_duality}, we have
\[
\sup_{m,n}b_n^{(m)}(\tau_0)\le b_0^{(0)}(\tau_0)=\sinh t_0\,b_0^{(0)}(t_0)\le C\tau_0e^{-\tau_0}.
\]
Since the diagonal entries of $T_{\tau_0}$ are $1$, we have $\widetilde K^{(m)}_{\tau_0}-\operatorname{diag}(\beta_n^m)=\frac12\sinh\tau_0\,\Acal^{(m)}_{\tau_0}(T_{\tau_0}-\Id)\Acal^{(m)}_{\tau_0}$, and $\|T_{\tau_0}-\Id\|=\coth(\tau_0/2)-1=2/(e^{\tau_0}-1)$ by \eqref{E_symbol}. Hence
\begin{equation}\label{slender_offdiagonal_bound}
\sup_m\bigl\|\widetilde K_{\tau_0}^{(m)}-\operatorname{diag}(\beta_n^m(\tau_0))\bigr\|
\le\frac{1+e^{-\tau_0}}{2}\sup_{m,n}b_n^{(m)}(\tau_0)\le C\tau_0e^{-\tau_0}.
\end{equation}

We next estimate the diagonal entries. In \eqref{diagonal_duality} with $t=\tau_0$, the coefficient $d_m^n(t_0)$ has the cross-sectional index $m$ and the azimuthal index $n$. By Lemma~\ref{lem:green_weight_bounds} at the parameter $t_0$, it is nonpositive unless $m=0\ne n$. For $m=0\ne n$, \eqref{positive_diagonal_rank_one} at the parameter $t_0$, together with $\iota(t_0)=\tau_0$, gives
\[
d_0^n(t_0)\le\frac{C}{n^2\tau_0^2}.
\]
The last term in \eqref{diagonal_duality} is nonpositive and its absolute value is bounded by $Ct_0^2(1+|\log t_0|)\le C\tau_0e^{-2\tau_0}$ by Lemma~\ref{lem:global_weight_bound}. Hence $\beta_n^m(\tau_0)\ge-C\tau_0^{-2}$ for all $m,n$. From \eqref{slender_offdiagonal_bound}, it follows that
\[
\bigl\langle\widetilde K_{\tau_0}^{(m)}u,u\bigr\rangle\ge-\frac{C}{\tau_0^2}\|u\|^2,\qquad u\in\ell^2(\ZZ),
\]
where $C$ is independent of $m$. By summing over the orthogonal decomposition, we have the lower bound in \eqref{slender_ring_negative_bound}. The upper bound holds since $0\in\sigma(\Kcal^*_{\partial\Omega_0})$ by the compactness. As a consequence, every point of $\sigma(\Kcal^*_{\partial\Omega_0})$ lies within distance $C\tau_0^{-2}$ of $[0,1/2]$.

It remains to show that every point of $[0,1/2]$ is approximated by the spectrum. Fix $s>0$ and let $m=\lfloor s/t_0\rfloor$, so that $|mt_0-s|<t_0$ and $m\ge1$ for large $\tau_0$. By Proposition~\ref{prop:D_asym} at the parameter $t_0$, with the azimuthal index $0$ and the cross-sectional index $m$, and by the mean value theorem, we have
\[
b_m^{(0)}(t_0)=I_0K_0(s)+O_s(t_0),\qquad d_m^0(t_0)=D_0(s)+O_s(t_0).
\]
From \eqref{diagonal_duality} with $n=0$, we have $|\beta^m_0(\tau_0)+D_0(s)|\le C_st_0$, and hence, by \eqref{slender_offdiagonal_bound},
\[
\bigl\|(\Kcal^*_{\partial\Omega_0}+D_0(s))U_me_0\bigr\|
=\bigl\|(\widetilde K^{(m)}_{\tau_0}+D_0(s))e_0\bigr\|
\le C_s\tau_0e^{-\tau_0}.
\]
Since $\operatorname{dist}(\lambda,\sigma(T))\le\|(T-\lambda)u\|$ for a self-adjoint operator $T$ and a unit vector $u$ (see \cite[\S3.1]{Teschl2014}), we have $\operatorname{dist}(-D_0(s),\sigma(\Kcal^*_{\partial\Omega_0}))\to0$ for every $s>0$. By Lemma~\ref{lem:diagonal_range}, $\{-D_0(s):s>0\}$ is dense in $[0,1/2]$. Since $\lambda\mapsto\operatorname{dist}(\lambda,\sigma)$ is $1$-Lipschitz, the convergence at finitely many points of an $\varepsilon$-net of $[0,1/2]$ gives $\sup_{\lambda\in[0,1/2]}\operatorname{dist}(\lambda,\sigma(\Kcal^*_{\partial\Omega_0}))\to0$. This completes the proof.
\end{proof}

\section{Numerical examples}
\label{sec:numerical_validation}

We compute finite sections of \eqref{similar_block} for tori with $R_0=1$ and $r_0=\operatorname{sech}\tau_0$, some of which are shown in Figure~\ref{fig:tori}, using $|m|\le6$ and $|n|\le\lceil30/\tau_0\rceil$ for $\tau_0=0.01,0.05,0.1,0.5$, and $|n|\le24$ and $|m|\le\lceil12\sinh\tau_0\rceil$ for $\tau_0=1,2,4,8$. The coefficients are computed by the recurrence relations of the Legendre functions, together with Lemma~\ref{lem:toroidal_duality} and numerical quadrature for large $m$. In our tests, enlarging the cutoffs or using high-precision arithmetic changed the computed eigenvalues by less than $10^{-11}$.

Figure~\ref{fig:two_limits} illustrates Theorems~\ref{thm:horn_full_limit} and \ref{thm:slender_ring_limit}. As $\tau_0\downarrow0$, the negative eigenvalues fill $[-1/2,0]$, while the largest positive eigenvalues are already close to their limits. As $\tau_0\to\infty$, the negative eigenvalues approach $0$ and the positive ones fill $[0,1/2]$.

\begin{figure}[H]
\centering
\includegraphics[width=0.95\textwidth,trim=0 18bp 0 28bp,clip]{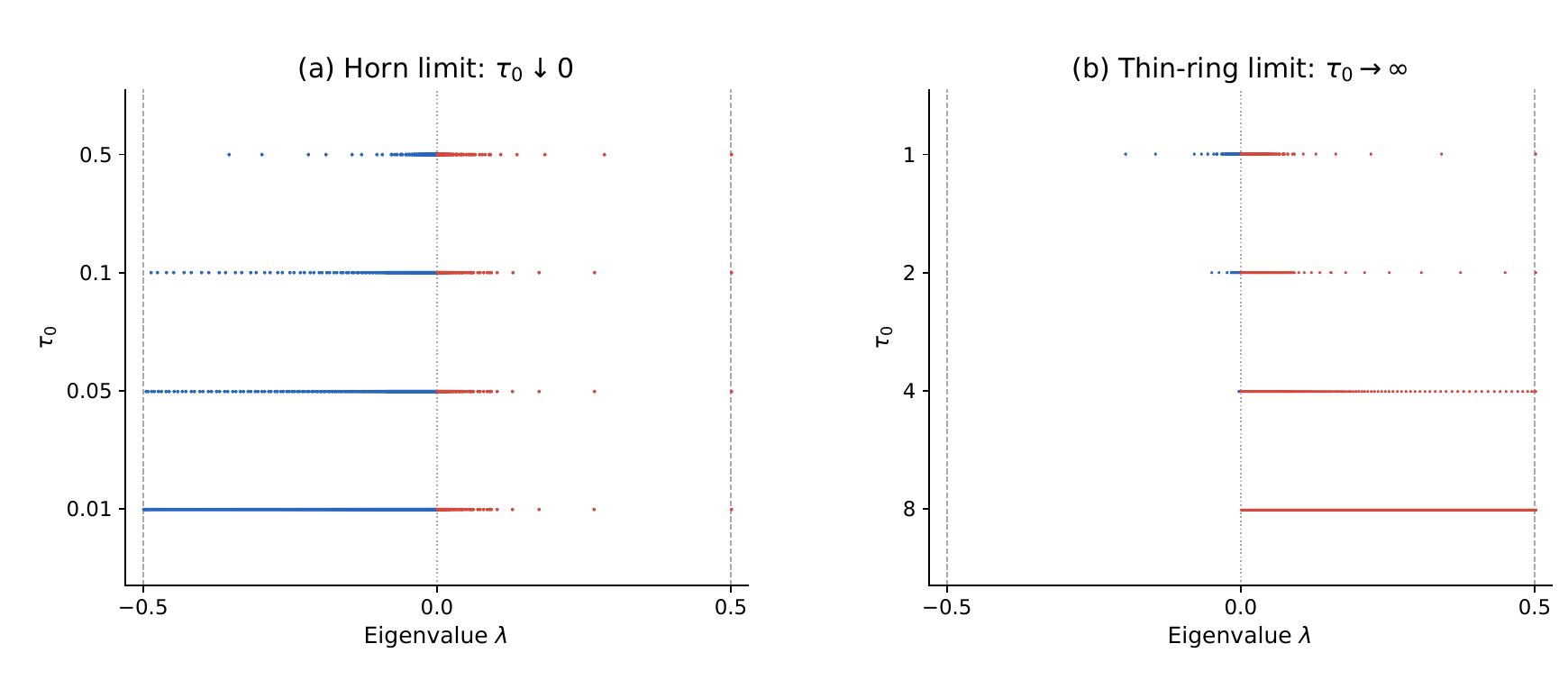}
\caption{Computed NP eigenvalues for tori with major radius one, with the Fourier cutoffs described in the text. Blue and red points indicate negative and positive eigenvalues, respectively, and the rows are labeled by $\tau_0$.}
\label{fig:two_limits}
\end{figure}

\begin{figure}[H]
\centering
\includegraphics[width=\textwidth,trim=0 0 0 25bp,clip]{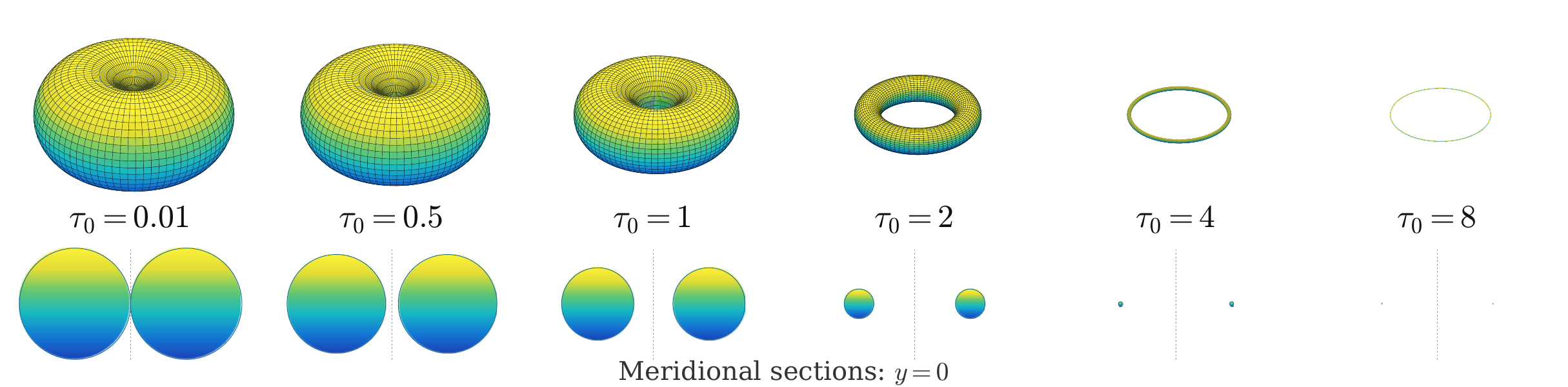}
\caption{The solid tori for $\tau_0=0.01,0.5,1,2,4,8$ and their sections by the plane $y=0$.}
\label{fig:tori}
\end{figure}

\appendix

\section{Proof of Lemma~\ref{lem:green_weight_bounds}}
\label{app:weight_bounds}

\subsection{Comparison of solution products}
\label{app:weight_equations}

For fixed $m,n$, let $p(t)=P^{-|m|}_{|n|-1/2}(\cosh t)$ and $q(t)=(-1)^{|m|}Q^{|m|}_{|n|-1/2}(\cosh t)$, so that $b_n^{(m)}(t)=p(t)q(t)$. Note that $p>0$ and $q>0$, since $pq>0$ by \eqref{density_gram_matrix}. By substituting $x=\cosh t$ into the associated Legendre equation \cite[Eq.~14.2.2]{DLMF}, we have
\begin{equation}\label{app:radial_equation}
w''+\coth t\,w'-\Bigl(n^2-\frac14+m^2\operatorname{csch}^2t\Bigr)w=0,
\end{equation}
and \eqref{normalized_wronskian} becomes $pq'-p'q=-1/\sinh t$. We transform \eqref{app:radial_equation} into two Schr\"odinger-type equations.

First, the functions $\sqrt{\sinh t}\,p$ and $\sqrt{\sinh t}\,q$ satisfy
\begin{equation}\label{app:first_equation}
-u''+\Bigl[\Bigl(m^2-\frac14\Bigr)\operatorname{csch}^2t+n^2\Bigr]u=0,
\end{equation}
their Wronskian is $-1$, and their product is $\sinh t\,b_n^{(m)}(t)$. Second, we set $p_\iota(y)=p(\iota(y))$ and $q_\iota(y)=q(\iota(y))$. Since $dt/dy=-\sinh t$ and $\sinh t=\operatorname{csch}y$ for $t=\iota(y)$, these functions satisfy
\begin{equation}\label{app:second_equation}
-u''+\Bigl[\Bigl(n^2-\frac14\Bigr)\operatorname{csch}^2y+m^2\Bigr]u=0,
\end{equation}
with $q_\iota p_\iota'-q_\iota'p_\iota=-1$ and $p_\iota q_\iota=b_n^{(m)}\circ\iota$.

From \cite[Eqs.~14.3.6--14.3.7, 14.3.10, 14.8.7, 14.8.15]{DLMF}, we have the following endpoint behavior, where $c$ denotes positive constants. In \eqref{app:first_equation}, $\sqrt{\sinh t}\,p\sim c\,t^{|m|+1/2}$ as $t\to0$, and $\sqrt{\sinh t}\,q\sim ce^{-|n|t}$ as $t\to\infty$ if $n\ne0$, while it tends to a positive constant if $n=0$. In \eqref{app:second_equation}, $q_\iota\sim c\,y^{|n|+1/2}$ as $y\to0$, without a logarithmic term even for $n=0$, and $p_\iota\sim ce^{-|m|y}$ as $y\to\infty$, which tends to a positive constant if $m=0$. We call the first solution of each equation the solution selected at zero, and the second one the solution selected at infinity.

We now state a comparison principle for these products. For $i=1,2$, let $u_i,v_i>0$ satisfy $w''=V_iw$ on $(0,\infty)$ and $u_iv_i'-u_i'v_i=-1$. We assume that $V_1\ge V_2$, that $(V_1-V_2)u_1u_2$ is integrable near $0$, that $(V_1-V_2)v_1v_2$ is integrable near $\infty$, and that
\[
(u_1'u_2-u_1u_2')(t)\to0\ \ (t\downarrow0),\qquad
(v_1'v_2-v_1v_2')(t)\to0\ \ (t\to\infty).
\]
Since $(u_1'u_2-u_1u_2')'=(V_1-V_2)u_1u_2$ and similarly for $v_i$, we have
\begin{equation}\label{app:log_derivative_comparison}
\frac{u_1'}{u_1}-\frac{u_2'}{u_2}=\frac{\int_0^x(V_1-V_2)u_1u_2}{u_1u_2}(x)\ge0,
\qquad
\frac{v_1'}{v_1}-\frac{v_2'}{v_2}=-\frac{\int_x^\infty(V_1-V_2)v_1v_2}{v_1v_2}(x)\le0.
\end{equation}
From the normalization, we have $1/(u_iv_i)=u_i'/u_i-v_i'/v_i$, and hence
\begin{equation}\label{app:product_comparison_bound}
u_1(x)v_1(x)\le u_2(x)v_2(x),\qquad x>0.
\end{equation}

In all comparisons below, the potentials satisfy $V(t)=(j^2-\frac14)t^{-2}+O(1)$ near $0$ for some integer $j\ge0$, and the solutions selected at zero are of the form $u=t^{j+1/2}h$ with $h\to c>0$. By integrating $(t^{2j+1}h')'=O(t^{2j+1})$, we have $h=c+O(t^2)$ and $h'=O(t)$. Hence the cross-Wronskian at $0$ is $O(t^{a_1+a_2-1})$ if the leading exponents $a_1,a_2\ge\frac12$ are not both $\frac12$, and it is $t(h_1'h_2-h_1h_2')=O(t^2)$ otherwise. At infinity, the selected solutions either decay exponentially or tend to positive constants, in which case the potential is $O(e^{-2t})$. In both cases, $v_i'\to0$, and the cross-Wronskian tends to $0$. The integrability conditions can be checked directly in each case.

\subsection{Bounds for the weights and the sign of the diagonal coefficients}

For the constant potential $\zeta^2$ with $\zeta>0$, the normalized solutions are $\sinh(\zeta t)/\zeta$ and $e^{-\zeta t}$, and their product is bounded by $1/(2\zeta)$. For $m,n\ne0$, the potential in \eqref{app:first_equation} is bounded below by $n^2$. Hence, by \eqref{app:product_comparison_bound}, we have $\sinh\tau_0\,b_n^{(m)}(\tau_0)\le1/(2|n|)$, which proves \eqref{weight_n_decay}. Similarly, from \eqref{app:second_equation}, we have $b_n^{(m)}\le1/(2|m|)$.

If the coefficient of $\operatorname{csch}^2$ is $-\frac14$, we use $\zeta^2-\frac14\operatorname{csch}^2t\ge\zeta^2-\frac1{4t^2}$. The normalized solutions for the right-hand side are $\sqrt t\,I_0(\zeta t)$ and $\sqrt t\,K_0(\zeta t)$ (see \cite[Eqs.~10.25.1, 10.28.2]{DLMF}), and $V_1-V_2=\frac14(t^{-2}-\operatorname{csch}^2t)$ is bounded. By applying this comparison to \eqref{app:first_equation} with $m=0$, we have \eqref{weight_zero_mode_decay}. Similarly, from \eqref{app:second_equation} with $n=0$, we have
\begin{equation}\label{app:exceptional_weight}
b_0^{(m)}(\tau_0)\le\iota(\tau_0)\,I_0K_0\bigl(|m|\iota(\tau_0)\bigr),\qquad m\ne0.
\end{equation}
Since $sI_0K_0(s)$ is bounded on $(0,\infty)$ by subsection~\ref{subsec:endpoint_behavior}, the right-hand side is bounded by $C/|m|$, which proves \eqref{weight_m_decay}.

We next show that $d_n^m\le0$ for $n\ne0$. Let $n\ne0$. The potential $V$ in \eqref{app:second_equation} is decreasing in $y$. For $h>0$, the shifted pair $(q_\iota(\cdot+h),p_\iota(\cdot+h))$ satisfies the equation with the smaller potential $V(\cdot+h)$, and its Wronskian is $-1$. Since $|n|\ge1$, both $q_\iota$ and $q_\iota'$ vanish at $0$, while the shifted functions are bounded near $0$. Hence the cross-Wronskian at $0$ vanishes, and $(V(y)-V(y+h))q_\iota(y)q_\iota(y+h)=O(y^{|n|-3/2})$ is integrable. By \eqref{app:product_comparison_bound}, we have $p_\iota q_\iota(y)\le p_\iota q_\iota(y+h)$, that is, $b_n^{(m)}\circ\iota$ is nondecreasing. Since $d\iota/dt=-1/\sinh t$, we have
\begin{equation}\label{app:d_as_y_derivative}
d_n^m(\tau_0)=-\frac12\,\frac{d}{dy}\bigl(p_\iota q_\iota\bigr)\Big|_{y=\iota(\tau_0)},
\end{equation}
and hence $d_n^m\le0$.

We also show that $d_0^0<0$. By differentiating \eqref{weights_fourier_integral} with $m=n=0$, we have
\[
\frac{d}{d\tau_0}b_0^{(0)}(\tau_0)
=-\frac{\sinh\tau_0\cosh\tau_0}{4\pi\sqrt2}
\int_{-\pi}^{\pi}\!\!\int_{-\pi}^{\pi}
\frac{(1-\cos\phi)\,d\phi\,d\theta}
{\bigl[1-\cos\theta+\sinh^2\tau_0(1-\cos\phi)\bigr]^{3/2}}<0,
\]
where the differentiation under the integral sign is justified since the differentiated integrand is $O((\phi^2+\theta^2)^{-1/2})$ near the origin, locally uniformly in $\tau_0$. Hence $d_0^0<0$.

\subsection{Bound for $d_0^m$}
\label{app:d0m}

We now prove \eqref{positive_diagonal_rank_one}. Let $n=0\ne m$, and set $A=p_\iota'/p_\iota$. Since the potential $m^2-\frac14\operatorname{csch}^2y$ in \eqref{app:second_equation} is bounded above by $m^2$, the comparison of $p_\iota$ with $e^{-|m|y}$ in \eqref{app:log_derivative_comparison} gives $A\ge-|m|$. From the Wronskian, we have $(p_\iota q_\iota)'=2A\,p_\iota q_\iota+1$, and hence, by \eqref{app:d_as_y_derivative},
\[
d_0^m=-A\,b_0^{(m)}(\tau_0)-\frac12\le|m|\,b_0^{(m)}(\tau_0)-\frac12 .
\]
From \eqref{app:exceptional_weight} and $sI_0K_0(s)=\frac12+O(s^{-2})$, it follows that
\begin{equation}\label{app:exceptional_diagonal_bound}
d_0^m\le|m|\iota(\tau_0)\,I_0K_0\bigl(|m|\iota(\tau_0)\bigr)-\frac12\le\frac{C}{m^2\iota(\tau_0)^2},
\qquad 0<\tau_0\le1,
\end{equation}
where we used $\iota(\tau_0)\ge\iota(1)>0$. This proves \eqref{positive_diagonal_rank_one}.

This completes the proof of Lemma~\ref{lem:green_weight_bounds}.

\end{document}